\documentclass[a4paper,11pt]{article}

\title{Generation of Generalised Wreath Products of Symmetric Groups}
\author{Jiaping Lu \\[10pt]
  Mathematical Institute, University of St Andrews,\\
  North Haugh, St Andrews, Fife, KY16 9SS\\[10pt]
  \begin{tabular}{c}\texttt{jl337@st-andrews.ac.uk}\\
    \texttt{jiapingljp@hotmail.com}
  \end{tabular}}

\usepackage{amsmath,amssymb}
\usepackage[margin=2.9cm]{geometry}
\usepackage[amsmath,thmmarks]{ntheorem}
\usepackage{enumitem}
\usepackage{tikz}
\usepackage{mathtools}
\usepackage{booktabs}

\theorembodyfont{\slshape}
\newtheorem{lem}{Lemma}[section]

\newtheorem{cor}[lem]{Corollary}
\newtheorem{thm}[lem]{Theorem}
\theorembodyfont{\normalfont}

\newtheorem{defn}[lem]{Definition}

\theorembodyfont{\slshape}
\theoremnumbering{Alph}

\makeatletter
\theoremstyle{nonumberplain}
\theorembodyfont{\normalfont}
\theoremheaderfont{\normalfont\scshape}
\theoremsymbol{~\ensuremath\square}
\newtheoremstyle{proofstyle}%
  {\item[\theorem@headerfont\hskip\labelsep ##1\theorem@separator]}%
  {\item[\theorem@headerfont\hskip\labelsep ##3\theorem@separator]}
\theoremstyle{proofstyle}
\newtheorem{prf}{Proof:}
\makeatother

\renewcommand{\leq}{\leqslant}
\renewcommand{\geq}{\geqslant}
\newcommand{\nbd}{\nobreakdash-}

\newcommand{\trivsubgp}{\mathbf{1}}

\newcommand{\order}[1]{\mathopen{|}#1\mathclose{|}}

\newcommand{\set}[2]{\{\,#1\mid#2\,\}}
\newcommand{\Sym}{\operatorname{Sym}}

\renewcommand{\geq}{\geqslant}
\renewcommand{\leq}{\leqslant}

\renewcommand{\wr}{\operatorname{wr}}
\newcommand{\Jrel}{\underset{J}{\sim}}

\setlist[enumerate,1]{label={\normalfont(\roman*)}}

\usepackage{hyperref}
\hypersetup{colorlinks=true, linkcolor={red!50!black},
  citecolor={green!50!black}}

\begin{document}

\maketitle

\begin{abstract}
Let $I$ be a finite partially ordered set and let $(\Sym(\Delta_i),\Delta_i)_i$ be a sequence of symmetric groups indexed by $I$. Construct the generalised wreath product $(F,\Delta)$ on this sequence of permutation groups. We determine the minimum number $d(F)$ of generators required for this generalised wreath product.
\end{abstract}

\paragraph{Keywords:} generating sets, wreath products, finite groups

\section{Introduction}
Given a finite group $G$, one common way to describe it is to use a generating set of $G$. Every element of $G$ can be expressed as a product of generators. There have been numerous results on the minimum number of generators of groups. For example, it is often covered in undergraduate group theory courses that any symmetric group can be generated by two elements. Also, it is also proved that every non-abelian finite simple group can be generated by two elements using Classification of Finite Simple Groups.

Another theme of research on generators of groups is the growth rate of the minimum number of generators of products of groups. Wiegold shows in \cite{Wiegold} that as $n$ increases, $d(G^n)$ grows linearly when $G$ is a soluble group and grows logarithmically when $G$ is a perfect group. Apart from direct products, there are also results on generation of iterated wreath products. Quick \cite{MRQ2} considers iterated wreath products constructed using any non-abelian finite simple groups with any faithful transitive actions. He proves that the probability that the iterated wreath products can be generated by two elements converges to the probability that the first wreath factor can be generated by two elements as the order of the first wreath factor tends to infinity. This then implies that this iterated wreath product can be generated by two elements. Bonderenko \cite{Bond} turns to iterated wreath products $W_n$ of finite transitive groups $G_i$ with uniformly bounded number of generators. He shows that  as $n$ tends to infinity, $W_n$ is finitely generated if and only if the direct product $G_1/G_1'\times G_2/G_2'\times\dots\times G_n/G_n'$ is finitely genearted. 

The crown-based product is another interesting product of groups, and it is useful to investigation of generation of groups. Dalla Volta and Lucchini prove in \cite{LuDallaVolta2} that if a finite group needs more generators than any of its proper quotients, then this finite group can be expressed as a crown-based product. Recently, Lucchini and Thakkar use crown-based products and establish an upper bound for the number of tests to determine a minimum generating set of a finite group.

Bailey et al. \cite{Generalised} introduce generalised wreath products constructed on a sequence of permutation groups indexed by a partially ordered set. This product includes direct products and wreath products as well as composition of these two products. When the partially ordered set is an antichain, the generalised wreath product is a direct product, while when the poset is a chain, the generalised wreath product is an iterated wreath product. In some cases, a crown-based product can be constructed using composition of direct products and wreath products, and so it can be expressed as a generalised wreath product. Bailey et al. also introduce poset block structures, and they prove that the group of automorphisms of a poset block structure is a generalised wreath product constructed on a sequence of symmetric groups. Moreover, Anagnostopoulou-Merkouri et al. \cite{Ana} proves that a transitive group which preserve a poset block structure is embedded in a generalised wreath product, and they determine the factors of this product.

In this paper, we will establish the minimum number of generators of generalised wreath products constructed on a sequence of symmetric groups indexed by arbitrary finite partially ordered sets using their natural actions. When the partially ordered set is a chain, we recover the growth rate of generators of iterated wreath product of symmetric groups determined in \cite{LuQuick}.  

Given a group $G$, denote by $d(G)$ the minimum number of generators of $G$. The theorem we will prove is as follows:
\begin{thm}\label{thm:chap4main1}
Let $(I, \leq)$ be a finite partially ordered set with $\order{I}\geq2$. For each $i\in I$, let $\Delta_i$ be a finite non-singleton set and let $(F, \Delta)$ be the generalised wreath product constructed from the symmetric groups $(\Sym(\Delta_i), \Delta_i)_{i\in I}$. Then
\[
d(F)=\order{I}.
\] 
\end{thm}

A summary of the paper follows. In Section~\ref{sec:Generalised}, we will introduce generalised wreath products using definitions and results in \cite{Generalised}. We will further investigate the structure of generalised wreath products in Section~\ref{sec:structure}. Finally, we will give the proof of Theorem~\ref{thm:chap4main1} in Section~\ref{sec:proof}.

\paragraph{Acknowledgements:} The author is funded by the China
Scholarship Council.  

\section{Generalised wreath products}\label{sec:Generalised}

In this section, we will introduce generalised wreath products as defined by Bailey et al. in \cite{Generalised}, and we will mostly inherit their notations.

Let $(I,\leq)$ be a non-empty partially ordered set. If two elements $i, j$ are incomparable in $I$, then we write $i\perp j$. A subset $J$ of $I$ is \textit{ancestral} if whenever $i>j$ and $j\in J$, then $i\in J$. For any $i\in I$, we define 
\[
A(i)=\{j\in I: j>i\}, 
\]
which is an ancestral subset of $I$.

Let $(\Delta_i)_{i\in I}$ be a sequence of non-empty sets indexed by elements of $I$. Set $\Delta=\prod_{i\in I}\Delta_i$, and $\Delta_J=\prod_{i\in J}\Delta_i$ for $J\subseteq I$. In particular, for $i\in I$, 
\[
\Delta_{A(i)}=\prod_{j\in A(i)}\Delta_i=\prod_{j>i}\Delta_i.
\]
For subsets $K\subseteq J\subseteq I$, denote by $\pi^J_K:\Delta_J\to \Delta_K$ the natural projection from $\Delta_J$ onto $\Delta_K$. If $K=\{k\}$, we shall write $\pi^J_k$ for $\pi^J_K$. We will also abbreviate $\pi^I_J$ to $\pi_J$.

We write elements in $\Delta$ as $\delta=(\delta_i)_{i\in I}$, where $\delta_i\in\Delta_i$ for each $i$. For $J\subseteq I$, we define an equivalence relation $\Jrel$ on $\Delta$ by $\gamma\Jrel\delta$ if and only if $\gamma\pi_J=\delta\pi_J$.

Now for each $i\in I$, let $G_i$ be a permutation group on $\Delta_i$ with the identity denoted by $1_i$, and let $F_i$ be the set of all functions from $\Delta_{A(i)}$ into $G_i$. For $J\subseteq I$, put $F_J=\prod_{i\in J}F_i$, and let $F=F_I$. Then $F$ is the \textit{generalised wreath product} constructed from $(G_i,\Delta_i)_{i\in I}$. We write elements of $F$ as $f=(f_i)_{i\in I}$ with $f_i\in F_i$.  If $K\subseteq J\subseteq I$, then denote by $\varphi^J_K: F_J\to F_K$ the natural projection. Analogously, if $K=\{k\}$, we will use $\varphi^J_k$ instead. We will also write $\varphi_J: F\to F_J$ for the natural projection.

In the following, we will show that a generalised wreath product is a group. But first observe that for $\delta\in\Delta$ and $i\in I$, $\delta\pi_{A(i)}\in\Delta_{A(i)}$. The set $F_i$ contains all functions from $\Delta_{A(i)}$ to the permutation group $(G_i, \Delta_i)$. Thus, $\delta\pi_{A(i)}f_i$ is an element of the group $G_i$, and we can apply it to $\delta_i\in\Delta_i$ using the group action of $G_i$ on $\Delta_i$. We shall denote the image by $\delta_i(\delta\pi_{A(i)}f_i)$. This induces maps from elements of $F$ on $\Delta$ as follows.

\begin{defn}\label{def:Fdef}
For $f=(f_i)_{i\in I}\in F$ and $\delta=(\delta_i)_{i\in I}\in \Delta$, define a function $\alpha: \Delta\times F\to \Delta$ as follows:
\[
(\delta, f)\mapsto(\delta_i(\delta\pi_{A(i)}f_i))_{i\in I}.
\]
\end{defn}
We will write $\delta f$ for the sequence $(\delta_i(\delta\pi_{A(i)}f_i))_{i\in I}$ of $\Delta$. The following lemma shows that for an ancestral subset $J\subseteq I$, the functions derived from elements of $F$ preserve the equivalence relation $\Jrel$.

\begin{lem}[Bailey et al.~{\cite[Lemma 2]{Generalised}}]\label{lem:relationpreserve}
Let $J$ be an ancestral subset of $I$. Let $\gamma,\delta\in \Delta$ and $f\in F$. If $\gamma\Jrel\delta$, then $\gamma f\Jrel\delta f$.
\end{lem}

We observe that if $J$ is an ancestral subset of $I$ and $j$ is an element of $J$, then the set $A(j)$ is the same whether we compute it in the whole set $I$ or within the subset $J$. Then, in either case, $\Delta_{A(j)}$ is identical, and so is the set $F_j$ of all functions from $\Delta_{A(j)}$ into $G_j$. Hence $F_J$ is the same object whether we treat $F_J=F\varphi_J$ or consider $F_J$ as the generalised wreath product constructed from $(G_i, \Delta_i)_{i\in J}$. Then elements of $F_J$ induce maps on $\Delta_J$ according to Definition~\ref{def:Fdef}. Let $\delta\in \Delta$. For $f\in F$ and an ancestral subset $J$, the image $f\varphi_J$ is an element of $F_J$ and we observe that
\begin{equation}\label{eq:fphiJ}
(\delta\pi_J)f\varphi_J=(\delta f)\pi_J.
\end{equation}
Also, if $K\subseteq J\subseteq I$ and $K,J$ are ancestral subsets of $I$, then we have a generalised version of Equation~\eqref{eq:fphiJ}:
\begin{equation}\label{eq:generalhat=phi=pi}
(\delta\pi^J_K)f\varphi^J_K=(\delta f)\pi^J_K.
\end{equation}

In the rest of this chapter, for ancestral subsets $K, J$ with $K\subseteq J$ and $f\in F_J$, in addition to $f\varphi^J_K$, we will also use $f_K$ for the image of $f$ in $F_K$.

Bailey et al. defined a binary operation on $F$ in \cite[Lemma 4]{Generalised} as follows, and they showed that $F$ is a group under this operation.

\begin{defn}\label{def:mul}
Let $f=(f_i)_{i\in I},h=(h_i)_{i\in I}$ be elements of $F$. Then we set $t=(t_i)_{i\in I}=fh$ with
\[
t_i=f_i(f_{A(i)}h_i)\in F_i,
\]
where $f_{A(i)}=f\varphi_{A(i)}\in F_{A(i)}$, and the product of the functions $f_i$ and $f_{A(i)}h_i$ from $\Delta_{A(i)}=\prod_{j\in A(i)}\Delta_j$ to $G_i$ is defined pointwise.
\end{defn}
In the formula above, we note that $f_{A(i)}$ induces a function on $\Delta_{A(i)}$. Then for $\omega\in\Delta_{A(i)}$, the pointwise product of $f_i$ and $f_{A(i)}h_i$ means 
\begin{equation}\label{eq:omegati}
\omega t_i=(\omega f_i)(\omega f_{A(i)}h_i).
\end{equation}
Now let $\delta=(\delta_i)_{i\in I}\in\Delta$, $f,h\in F$ and $t=fh$. Using Definition~\ref{def:Fdef} and Equation~\eqref{eq:omegati}, we have the expression of $\delta t$:
\[
\delta t=(\delta_i(\delta\pi_{A(i)}t_i))_{i\in I}=(\delta_i(\delta\pi_{A(i)}f_i)((\delta\pi_{A(i)}f_{A(i)})h_i))_{i\in I}.
\]
In the formula above, we observe that $\delta\pi_{A(i)}\in\Delta_{A(i)}$ and $f_i, h_i$ are functions $\Delta_{A(i)}\to G_i$. Then $\delta\pi_{A(i)}f_i$ and $(\delta\pi_{A(i)}f_{A(i)})h_i$ are elements of the permutation group $G_i$. Therefore, $\delta_i(\delta\pi_{A(i)}f_i)((\delta\pi_{A(i)}f_{A(i)})h_i)$ is the image of $\delta_i$ under the composition of the two permutations $\delta\pi_{A(i)}f_i$ and $(\delta\pi_{A(i)}f_{A(i)})h_i$.

The following lemma shows that there is an identity permutation on $\Delta$ in $F$ and we will denote it by $z$.

\begin{lem}[Bailey et al.~{\cite[Lemma 5]{Generalised}}]\label{lem:identity}
If, for each $i\in I$, the function $z_i\in F_i$ is defined by $\gamma z_i=1_i$ for all $\gamma\in \Delta_{A(i)}$ then $z=(z_i)_{i\in I}$ is the identity permutation on $\Delta$.
\end{lem}

The following theorem shows that if $I$ satisfies the ascending chain condition, then the generalised wreath product $(F,\Delta)$ constructed from permutation groups $(G_i, \Delta_i)_{i\in I}$ is a permutation group. It furthermore shows that, for any ancestral subset $J$ of $I$, $(F_J,\Delta_J)$ is a permutation group and $(\varphi_J, \pi_J)$ is a permutation homomorphism. We will then use $z_J$ for the identity of $F_J$.

\begin{thm}[Bailey et al.~{\cite[Theorem A]{Generalised}}]\label{thm:generalised}
Let $I$ be a partially ordered set with the ascending chain condition. Let $F$ be the generalised wreath product constructed from a sequence of permutation groups $(G_i, \Delta_i)_{i\in I}$. Then 
\begin{enumerate}[label=\textup{(\alph*)}]
\item\label{i:permutation} for all ancestral subsets $J$ of $I$, $(F_J, \Delta_J)$ is a faithful permutation group;
\item\label{i:homomorphism} if $J$ and $K$ are ancestral subsets of $I$ with $K\subseteq J$ then $(\varphi_K^J, \pi_K^J)$ is a permutation homomorphism from $(F_J, \Delta_J)$ onto $(F_K, \Delta_K)$ with kernel
\[
N_K^J=\set{f\in F_J}{f_j=z_j\ \text{for}\ j\in K}.
\]
\end{enumerate}
\end{thm}

Finally, we end this section with the following lemma, which establishes the condition for $(F,\Delta)$ to be transitive.
\begin{lem}[Bailey et al.~{\cite[Lemma 9]{Generalised}}]\label{lem:transitive}
The generalised wreath product of the permutation groups $(G_i,\Delta_i)_{i\in I}$ is transitive if and only if $(G_i,\Delta_i)$ is transitive for all $i\in I$.
\end{lem}

\section{The structure of generalised wreath products}\label{sec:structure}
Throughout, let $I$ be a finite partially ordered set. Let $(G_i,\Delta_i)_i$ be a sequence of finite permutation groups. Let $(F,\Delta)$ be the generalised wreath product with respect to this sequence. In this section, we will investigate the structure of $F$. We will first describe some subgroups and normal subgroups of $F$. For a minimal elemment $i$ of $I$, we will then show that $F$ decomposes as a split extension of a direct power $G_i^{\Delta_{A(i)}}$ by the generalised wreath product with respect to the poset $I\backslash\{i\}$. We will further show that $F$ can be expressed as an iterated wreath product constrcuted from the permutation groups $G_i$ when $I$ is a chain. When $I$ is an antichain, $F$ can be expressed as the direct product $\prod_{i\in I} G_i$. Finally, we will establish generating sets for $F$. 

We begin with key observations below.

\begin{lem}\label{lem:zitimeszi}
Let $f=(f_i)_{i\in I}, h=(h_i)_{i\in I}\in F$. Set $t=(t_i)_{i\in I}=fh$. Let $j\in I$.
\begin{enumerate} 
\item\label{i:fj=hj=zj}
Suppose that $f_j=h_j=z_j$. Then
\[
t_j=z_j.
\]
\item\label{i:fjinverse}
If $t_j=z_j$, then $f_j=z_j$ if and only if $h_j=z_j$.
\end{enumerate}
\end{lem}
\begin{prf}
\ref{i:fj=hj=zj}~Definition~\ref{def:mul} gives that $t_j=f_j(f_{A(j)}h_j)$. 
Let $\delta\in \Delta_{A(j)}$. Then $\delta t_j=(\delta f_j)((\delta f_{A(j)})h_j)$. As $f_j=h_j=z_j$, we have $\delta f_j=(\delta f_{A(j)})h_j=1_j$. The proof is completed.

\ref{i:fjinverse}~ Suppose that $h_j=z_j$. Let $\delta\in \Delta_{A(j)}$. Since $t_j=z_j$, we have $(\delta f_j)((\delta f_{A(j)})h_j)=\delta t_j=1_j$.
The fact that $(\delta f_{A(j)})h_j=1_j$ yields $\delta f_j=1_j$. As $\delta$ is arbitrary, $f_j=z_j$. 

Now assume that $f_j=z_j$. Write $f^{-1}=(d_i)_{i\in I}$ where $d_i\in F_i$ for each $i$. As $f^{-1}f=z$ is the identity of $F$, $d_j=z_j$ according to the result in the previous paragraph. Applying part~\ref{i:fj=hj=zj} to the formula $h=f^{-1}t$ yields that $h_j=z_j$. Then the proof is completed.
\end{prf}

We will then determine some subgroups and normal subgroups of $F$ in the following. For $j\in I$, let 
\[
L_j=\set{f=(f_{i})_{i\in I}\in F}{f_j=z_j}.
\]
It is a direct result of Lemma~\ref{lem:zitimeszi} that $L_j$ is a subgroup of $F$ for any $j\in I$.

Let $J$ be an ancestral subset of $I$. Set
\begin{equation}\label{eq:barFK}
\bar F_J=\bigcap_{k\in I\backslash J}L_k=\set{f=(f_i)_{i\in I}\in F}{f_i=z_i \ \text{for}\ i\notin J }.
\end{equation}
Also, for $i$ in $J$, set
\begin{equation}\label{eq:Hi}
H_i^J=\set{f=(f_j)_{j\in J}\in F_J}{f_j=z_j\ \text{for}\ j\neq i}.
\end{equation}
If $J=I$, then we write $H_i$ instead. In particular, 
\[
H_i=\set{f=(f_j)_{j\in I}\in F}{f_j=z_j\ \text{for}\ j\neq i}=\bigcap_{k\in I\backslash\{i\}}L_k.
\]
Recall from Section~\ref{sec:Generalised} that $\varphi_J$ is the projection map from $F$ to $F_J$. We give the following lemma without proof.
\begin{lem}\label{lem:barFKiso}
Let $J$ be an ancestral subset of $I$, and let $i\in J$. Then $\bar F_J\varphi_J=F_J$ and $H_i\varphi_J=H^J_i$. Moreover, $\varphi_J$ restricted on $\bar F_J$ induces an isomorphism from $\bar F_J$ to $F_J$.
\end{lem}

The subgroups $H_i$ play an important part in analysing the structure of $F$. Therefore, in the following, we will first focus on properties of $H_i$, and then we turn to the structure of $F$.

For a minimal element $i\in I$, set $J=I\backslash\{i\}$. Then $J$ is ancestral. By Theorem~\ref{thm:generalised}~\ref{i:homomorphism}, $H_i=N^I_J$ is normal in $F$. We have the following lemma.
\begin{lem}\label{lem:Hinormal} Let $i, j$ be two distinct minimal elements of $I$. Then
\begin{enumerate}
\item\label{i:Hinormal}
the subgroup $H_i$ is normal in $F$;
\item\label{i:HicrossHi}
$\langle H_i, H_j\rangle= H_i\times H_j$.
\end{enumerate}
\end{lem}

If $i$ is a minimal element in $I$ and $J=I\backslash\{i\}$. Then $H_i\cap\bar F_J$ is trivial and $\order{H_i\bar F_J}=\order{F}$. The fact that $H_i$ is a normal subgroup of $F$ by Lemma~\ref{lem:Hinormal} gives the following corollary immediately.

\begin{cor}\label{cor:FJsemi}
Let $i$ be a minimal element in $I$. Set $J=I\backslash\{i\}$. Then $F=H_i\rtimes \bar F_J$.
\end{cor}

For an ancestral subset $J$ and $i\in J$, we define a map $\theta_i^J:H^J_i\to G_i^{\Delta_{A(i)}}$ by:
\begin{equation}\label{eq:thetai}
f\theta_i^J=(\delta f_i)_{\delta\in \Delta_{A(i)}},
\end{equation}
where $G_i^{\Delta_{A(i)}}$ means $\order{\Delta_{A(i)}}$ copies of $G_i$ indexed by elements of $\Delta_{A(i)}$ and $f=(f_j)_{j\in J}\in H_i^J$. Note that $f_i$ is a function from $\Delta_{A(i)}$ to $G_i$. If $J=I$, we will write $\theta_i$ instead of $\theta^J_i$. The following lemma shows that $\theta_i$ is an isomorphism.
\begin{lem}\label{lem:Hidirectproduct}
For $i\in I$, the map $\theta_i$ as defined in Equation~\eqref{eq:thetai} is an isomorphism.
\end{lem}
\begin{prf}
Let $f=(f_j)_{j\in I}, h=(h_j)_{j\in I}\in H_i$. Then
\[
(f\theta_i )(h\theta_i)=\big((\delta f_i)_{\delta\in \Delta_{A(i)}}\big)\big((\delta h_i)_{\delta\in \Delta_{A(i)}}\big)=\big((\delta f_i)(\delta h_i)\big)_{\delta\in \Delta_{A(i)}},
\]
where $(\delta f_i)(\delta h_i)$ is a product in $G_i$. On the other hand, set $t=(t_j)_{j\in I}=fh$. Then
\[
t_i=f_i(f_{A(i)}h_i)
\]
according to Definition~\ref{def:mul}. For $\delta\in \Delta_{A(i)}$, $\delta t_i= (\delta f_i)( (\delta f_{A(i)})h_i)$. Observe that $f_{A(i)}=f\varphi_{A(i)}=z_{A(i)}$ is the identity of the permutation group $(F_{A(i)}, \Delta_{A(i)})$. Therefore, $(\delta f_{A(i)})h_i =\delta h_i$ and
\[
\delta t_i= (\delta f_i)( \delta(f_{A(i)}h_i))=(\delta f_i)(\delta h_i)
\]
for all $\delta\in \Delta_{A(i)}$. Hence
\[
(fh)\theta_i=(\delta t_i)_{\delta\in\Delta_{A(i)}}=((\delta f_i)(\delta h_i))_{\delta\in\Delta_{A(i)}}.
\]
It follows that $(f\theta_i)(h\theta_i)=(fh)\theta_i$ and $\theta_i$ is a homomorphism.

Now let $(g_{\delta})_{\delta\in \Delta_{A(i)}}$ be an element of $G_i^{\Delta_{A(i)}}$. Define the element $f\in F_i$ to be the function $\delta\mapsto g_{\delta}$ for each $\delta\in \Delta_{A(i)}$. Then define $h=(h_j)_{j\in I}$ by setting $h_i=f$ and $h_j=z_j$ for $j\in I\backslash\{i\}$. By construction, $h$ is an element of $H_i$ and by Equation~\eqref{eq:thetai},
\[
h\theta_i=(\delta f)_{\delta\in\Delta_{A(i)}}=(g_{\delta})_{\delta\in\Delta_{A(i)}}.
\]
Thus $\theta_i$ is surjective. 

Finally, let $f=(f_{j})_{j\in I}\in H_i$ such that $f\theta_i=(1_i,\dots,1_i)\in G_i^{\Delta_{A(i)}}$. Then $\delta f_i=1_i$ for all $\delta\in\Delta_{A(i)}$ by Equation~\eqref{eq:thetai}. Thus $f_i=z_i$, and consequently, $f=z$ by the definition of $H_i$. Hence $\theta_i$ is injective, and so $\theta_i$ is an isomorphism.
\end{prf}

For an ancestral subset $J$ of $I$ and an element $i\in I$ such that $A(i)\subseteq J$, the action of $F_{A(i)}$ on $\Delta_{A(i)}$ induces an action of $F_J$ on $G_i^{\Delta_{A(i)}}$ as follows
\begin{equation}\label{eq:actiononGDelta}
(g_{\delta})_{\delta\in \Delta_{A(i)}}^f=(g_{\delta (f^{-1}\varphi^J_{A(i)})})_{\delta\in \Delta_{A(i)}},
\end{equation}
where $(g_{\delta})_{\delta\in \Delta_{A(i)}}\in G_i^{\Delta_{A(i)}}$ and $f\in F_J$. The following lemma shows the interaction between this action and the isomorphism $\theta_i$.

\begin{lem}\label{lem:thetaandf}
Let $i$ be a minimal element in $I$. Let $J$ be the ancestral subset $I\backslash\{i\}$. If $f=(f_j)_{j\in I}\in \bar F_J$, then for $h\in H_i$,
\[
(h^f)\theta_i=(h\theta_i)^{f}.
\]
\end{lem}
In the lemma above, $h^f$ on the left-hand side means the conjugate of $h$ by $f$, while on the right-hand side we are using the action of $F$ on $G_i^{\Delta_{A(i)}}$ as shown in Equation~\eqref{eq:actiononGDelta}.

\begin{prf}
Let $h=(h_j)_{j\in I}\in H_i$. Set $d=(d_j)_{j\in I}=f^{-1}$, $t=(t_j)_{j\in I}=f^{-1}h$ and $u=(u_j)_{j\in I}=h^f=tf$. Using Definition~\ref{def:mul}, we obtain $t_i=d_i(d_{A(i)}h_i)$ and $u_i=t_i(t_{A(i)}f_i)$. As $f\in \bar F_J$, $f_i=z_i$. Then Lemma~\ref{lem:zitimeszi}\ref{i:fjinverse} gives $d_i=z_i$, and so $t_i=d_{A(i)}h_i$. Observe that for all $\delta\in\Delta_{A(i)}$,
\[
\delta u_i=(\delta t_i)((\delta t_{A(i)})f_i)=\delta t_i =\delta d_{A(i)}h_i.
\]
As $d_{A(i)}=f^{-1}\varphi_{A(i)}$, it follows that $\delta u_i=(\delta(f^{-1}\varphi_{A(i)}))h_i$. Then Equation~\eqref{eq:thetai} gives
\begin{equation}\label{eq:uthetai}
u\theta_i=(\delta u_i)_{\delta\in \Delta_{A(i)}}=((\delta (f^{-1}\varphi_{A(i)}))h_i)_{\delta\in\Delta_{A(i)}}.
\end{equation}
We will turn to $(h\theta_i)^f$ in the following. We have $h\theta_i=(\delta h_i)_{\delta\in \Delta_{A(i)}}$. Then set $(g_{\delta})_{\delta\in \Delta_{A(i)}}=(\delta h_i)_{\delta\in \Delta_{A(i)}}$. Therefore,
\[
(\delta (f^{-1}\varphi_{A(i)}))h_i=g_{\delta (f^{-1}\varphi_{A(i)})}
\]
for all $\delta\in \Delta_{A(i)}$. Then we obtain that 
\begin{equation}\label{eq:htof}
(h\theta_i)^{f}=((g_{\delta})_{\delta\in\Delta_{A(i)}})^{f}=(g_{\delta (f^{-1}\varphi_{A(i)})})_{\delta\in\Delta_{A(i)}}=((\delta (f^{-1}\varphi_{A(i)}))h_i)_{\delta\in\Delta_{A(i)}}
\end{equation}
By comparing Equations~\eqref{eq:uthetai} and \eqref{eq:htof}, $(h^f)\theta_i=(h\theta_i)^{f}$.
\end{prf}

For $i\in I$, let $\varepsilon_i\in\Delta_{A(i)}$ be a fixed element. By Lemma~\ref{lem:Hidirectproduct}, there exists a subgroup $D_i^J$ of $H_i^J$ such that
\begin{equation}\label{eq:DiJdefinition}
D_i^J\theta_i^J=\trivsubgp\times\dots\times\trivsubgp\times \underset{\mathclap{\substack{\uparrow \\ \text{$\varepsilon_i$\nbd coordinate}}}}{G_i}\times\trivsubgp\times\dots\times\trivsubgp\leq G_i^{\Delta_{A(i)}}.
\end{equation}
By Equation~\eqref{eq:thetai}, $D^J_i$ has the following form:
\begin{equation}\label{eq:Diset}
D_i^J=\set{f=(f_j)_{j\in J}}{\text{$f_j=z_j$ for $j\neq i$ and $\delta f_i=1_i$ for $\delta\neq\varepsilon_i$}}.
\end{equation}
In the case that $J=I$, we shall write $D_i$ for $D^I_i$. If $i$ is a maximal element of $J$, then $\Delta_{A(i)}$ is a singleton and hence $G_i^{\Delta_{A(i)}}=G_i$, and in this case, $D_i^J=H_i^J$. 

Since 
\[
G_i^{\Delta_{A(i)}}=\langle\trivsubgp\times\dots\times\trivsubgp\times \underset{\mathclap{\substack{\uparrow \\ \text{$\delta$\nbd coordinate}}}}{G_i}\times\trivsubgp\times\dots\times\trivsubgp: \delta\in\Delta_{A(i)}\rangle,
\]
and $H_i\theta_i= G_i^{\Delta_{A(i)}}$, by the definition of $D_i$ in Equation~\eqref{eq:DiJdefinition} and Lemma~\ref{lem:thetaandf}, we have the following generating set for $H_i$ if  $i$ is a minimal element in $I$ and the permutation group $(G_j,\Delta_j)$ is transitive for each $j\in I$.

\begin{lem}\label{lem:Digeneration}
Let $i$ be a minimal element in $I$. Set $J=I\backslash\{i\}$. Suppose that the permutation group $(G_j,\Delta_j)$ is transitive for each $j>i$. Then
\[
H_i=\langle (D_i)^f:f\in \bar F_J\rangle.
\]
\end{lem}
\begin{prf}
According to Lemma~\ref{lem:transitive}, $F_{A(i)}$ is transitive on $\Delta_{A(i)}$. It follows from Equation~\eqref{eq:actiononGDelta} that
\[
G_i^{\Delta_{A(i)}}=\langle (\trivsubgp\times\dots\times\trivsubgp\times \underset{\mathclap{\substack{\uparrow \\ \text{$\epsilon_i$\nbd coordinate}}}}{G_i}\times\trivsubgp\times\dots\times\trivsubgp)^f: f\in\bar F_J\rangle.
\]
By the definition of $D_i$ in Equation~\eqref{eq:DiJdefinition}, $G_i^{\Delta_{A(i)}}=\langle (D_i\theta_i)^f: f\in\bar F_J\rangle$.
Using $(D_i\theta_i)^f=(D_i^f)\theta_i$ from Lemma~\ref{lem:thetaandf} and the fact that $\theta_i: H_i\to G_i^{\Delta_{A(i)}}$ is an isomorphism, we deduce that $H_i=\langle (D_i)^f:f\in \bar F_J\rangle$.
\end{prf}

If $i$ is a minimal element and $J=I\backslash\{i\}$, then Corollary~\ref{cor:FJsemi} gives that $F=H_i\rtimes \bar F_J$. In the following, we will further show that $F$ can be expressed as the wreath product $G_i\wr_{\Delta_{A(i)}} F_J$, where the action of $F_J$ on $G_i^{\Delta_{A(i)}}$ is as given in Equation~\eqref{eq:actiononGDelta}. We will also describe the structure of $F$ when $I$ is a chain or an antichain. Finally, we will construct generating sets for $F$ using the subgroups $H_i$ and $D_i$.

There are isomorphisms $\varphi_J:\bar F_J\to F_J$ and $\theta_i:H_i\to G_i^{\Delta_{A(i)}}$. The purpose of the following lemma is to exploit these isomorphisms to convert the semi-direct product decomposition $F\cong H_i\rtimes \bar F_J$ to a wreath product decomposition $G_i\wr_{\Delta_{A(i)}} F_J$.

\begin{lem}\label{lem:Generalisediswreath}
Suppose that $i$ is a minimal element of $I$. Set $J=I\backslash\{i\}$. Then
\[
F\cong G_i\wr_{\Delta_{A(i)}} F_J
\]
where as in Equation~\eqref{eq:actiononGDelta}, for $f\in F_J$ and $(g_{\delta})_{\delta\in\Delta_{A(i)}}\in G_i^{\Delta_{A(i)}}$,
\[
(g_{\delta})_{\delta\in\Delta_{A(i)}}^f=(g_{\delta (f^{-1}\varphi^J_{A(i)})})_{\delta\in\Delta_{A(i)}}.
\]
Moreover, 
\begin{enumerate}
\item\label{i:unique}
if $i$ is the unique minimal element of $I$, then $A(i)=J$ and the action of $F_J$ on $G_i^{\Delta_{A(i)}}$ is faithful;
\item\label{i:maximal}
if $i$ is incomparable to all the rest of elements of $I$, then $F\cong G_i\times F_J$.
\end{enumerate}
\end{lem}
\begin{prf}
Corollary~\ref{cor:FJsemi} tells us that $F=H_i\rtimes \bar F_J$. For $(h, f)\in H_i\rtimes \bar F_J$ with $h\in H_i$ and $f\in \bar F_J$, let $\varphi: H_i\rtimes \bar F_J \to G_i\wr_{\Delta_{A(i)}} F_J$ be defined as follows: 
\[
(h, f)\varphi=(h\theta_i)(f\varphi_J).
\]
Since $\theta_k$ and $\varphi_J$ are isomorphisms, it suffices to prove that for all $h\in H_i$ and $f\in\bar F_J$, $(h\theta_i)^{f\varphi_J}=(h^f)\theta_i$. Lemma~\ref{lem:thetaandf} gives that for $f\in \bar F_J$, $(h^f)\theta_i=(h\theta_i)^{f}$. As the actions of $F$ and $F_J$ on $G_i^{\Delta_{A(i)}}$ are induced by $\varphi_{A(i)}$ and $\varphi^{J}_{A(i)}$ respectively as given in Equation~\eqref{eq:actiononGDelta}, it follows that for $f\in\bar F_J$, $(h\theta_i)^{f\varphi_J}=(h\theta_i)^{f}$. Hence we conclude that $\varphi$ is an isomorphism and
\[
F\cong G_i\wr_{\Delta_{A(i)}} F_J.
\]

\ref{i:unique} Furthermore, if $i$ is the unique minimal element in $I$, then $J=A(i)$. Then $F_J$ is a permutation group on $\Delta_{A(i)}$ by Theorem~\ref{thm:generalised}\ref{i:permutation}, and so the action of $F_J$ on $G_i^{\Delta_{A(i)}}$ is faithful.

\ref{i:maximal} Suppose that $i$ is incomparable to all the rest of elements of $I$. Then in this case, $A(i)$ is empty and $\Delta_{A(i)}$ is  a singleton. Hence $F_J$ acts trivially on $G_i^{\Delta_{A(i)}}$, and so $F\cong G_i\times F_J$.
\end{prf}

The following two conclusions are given without proof in \cite{Generalised}. We also state them without proof below. One can observe that they can be deduced by using Lemma~\ref{lem:Generalisediswreath} and applying induction.

\begin{lem}\label{lem:chain}
Suppose that a partially order set $I=\{i_1,i_2,\dots, i_n\}$ is a chain such that
\[
i_1<i_2<\dots <i_n.
\]
Then
\[
F\cong G_{i_1}\wr (G_{i_2}\wr(\dots\wr (G_{i_{n-1}}\wr G_{i_n})\dots)),
\]
where the iterated wreath product is constructed via the action of the permutation group $(G_i, \Delta_i)$ for each $i\in I$.
\end{lem}

\begin{lem}\label{lem:antichain}
Suppose that a partially ordered set $I$ is an antichain. Then 
\[
F\cong \prod_{i\in I}G_i.
\]
\end{lem}

In the next lemma, we will show that $F$ is generated by the collection of $H_i$ where $i$ ranges through $I$. Then using Lemma~\ref{lem:Digeneration}, provided that $(G_i,\Delta_i)$ is transitive for each $i\in I$, we will further show that $F$ can be generated by the collection of $D_i$.

\begin{lem}\label{lem:FHphi}
Let $I$ be a finite partially ordered set and let $F$ be the generalised wreath product with respect to a sequence of finite permutation groups $(G_i, \Delta_i)_{i\in I}$. Then
\[
F=\langle H_i: i\in I\rangle.
\]
\end{lem}
\begin{prf}
We will use induction on the cardinality of $I$ to prove the lemma. Suppose that $I=\{i\}$. Then, using Equation~\eqref{eq:Hi}, we deduce that $H_i=\set{f_i}{f_i\in F_i}= F$.

Now let $n\geq1$ be an integer. Suppose that the lemma holds for partially ordered sets of cardinality $n$. Then assume that $\order{I}=n+1$. Let $j$ be a minimal element in $I$, and let $J=I\backslash\{j\}$. Then $J$ is an ancestral subset and Corollary~\ref{cor:FJsemi} tells us that $F=H_j\rtimes \bar F_J$. It suffices to show that $\bar F_J=\langle H_i: i\in J\rangle$. Recall the definition of $H^J_i$ in Equation~\eqref{eq:Hi}. As $F_J$ is the generalised wreath product of the permutation groups $(G_i, \Delta_i)_{i\in J}$, according to the inductive hypothesis, $F_J=\langle H^J_i:i\in J\rangle$. By Lemma~\ref{lem:barFKiso}, $\varphi_J$ restricted on $\bar F_J$ is an isomorphism and $\bar F_J\varphi_J= F_J$. Then $H_i\varphi_J=H^J_i$ for $i\in J$, and so $\bar F_J=\langle H_i: i\in J\rangle$. The proof is completed.
\end{prf}

In the following lemma, we construct a generating set for $F$ with $D_i$.
\begin{lem}\label{lem:GeneralisedD}
Let $I$ be a finite partially ordered set and let $F$ be the generalised wreath product with respect to a sequence of finite transitive permutation groups $(G_i, \Delta_i)_{i\in I}$. Then
\[
F=\langle D_i:i\in I\rangle.
\]
\end{lem}
\begin{prf}
We will use induction on the cardinality of $I$ to prove the lemma. If $I=\{i\}$, then $D_i=H_i$ and so $D_i=F$.

Let $n\geq1$ be an integer. Suppose the lemma holds for partially ordered sets of cardinality $n$. Now assume that $\order{I}=n+1$. Let $j$ be a minimal element in $I$ and set $J=I\backslash\{j\}$. Corollary~\ref{cor:FJsemi} tells us that $F=H_j\rtimes \bar F_J$.
As $F_J$ is the generalised wreath product of the permutation groups $(G_i, \Delta_i)_{i\in J}$, it follows from the inductive hypothesis that $F_J=\langle D^J_i:i\in J\rangle$.
By Lemma~\ref{lem:barFKiso}, $\varphi_J$ restricted on $\bar F_J$ is an isomorphism and $\bar F_J\varphi_J=F_J$. Then $D_i\varphi_J=D^J_i$ for $i\in J$, and so $\bar F_J=\langle D_i:i\in J \rangle$. As each $G_i$ is transitive, Lemma~\ref{lem:Digeneration} establishes $
H_j=\langle (D_j)^f:f\in \bar F_J\rangle$. Hence
\[
F=H_j\rtimes \bar F_J=\langle D_i:i\in I\rangle.
\]
\end{prf}

\section{Proof of Theorem~\ref{thm:chap4main1}}\label{sec:proof}

In this section, we will prove Theorem~\ref{thm:chap4main1} and we will use induction on the cardinality of partially ordered sets. We will first prove that the theorem holds when the partially ordered set has two or three elements. Throughout this section, let $I$ be a partially ordered set. Let $(\Sym(\Delta_i),\Delta_i)_{i\in I}$ be a sequence of non-trivial symmetric groups. Denote by $F$ the generalised wreath product constructed from the sequence $(\Sym(\Delta_i),\Delta_i)_{i\in I}$. We will first determine the structure of $I$ when $\order{I}=2, 3$ and then establish $d(F)$ in this case.

If $\order{I}=2$, we will write $I=\{i,  j\}$, and if $I$ has exactly $3$ elements, then we will write $I=\{i,j,k\}$. We determine the structure of $I$ as follows without proof. 
\begin{lem}\label{lem:I}
Let $I$ be a partially ordered set such that $\order{I}=2, 3$. Then $I$ is one of the following up to isomorphisms:
\begin{enumerate}
\item\label{i:chain}
$I$ is a chain;
\item\label{i:antichain}
$I$ is an antichain;
\item\label{i:triangle}
$\order{I}=3$ and for $I=\{i,j,k\}$, $i<j$, $i<k$ and $j\perp k$;
\item\label{i:pyramid}
$\order{I}=3$ and for $I=\{i,j,k\}$, $i<k$, $j<k$ and $i\perp j$;
\item\label{i:wrdi}
$\order{I}=3$ and for $I=\{i,j,k\}$, $i<k$, $i\perp j$ and $j\perp k$.
\end{enumerate}
\end{lem}

Before we determine $d(F)$, we will first introduce the following result.

\begin{lem}\label{lem:symwreath}
If $H$ is a non-trivial symmetric group and $(G, X)$ is a transitive permutation group with $d(G)\geq2$, then
\[
d(H\wr_X G)=\max(2, d(C_2\wr G)).
\]
\end{lem}
\begin{prf}
If $H=S_n$ with $n\geq5$, then $A_n^X$ is the unique minimal normal subgroup of $H\wr G$. Using Theorem 1.1 in \cite{LuMe} yields that $d(H\wr G)=\max(2, d(C_2\wr G))$. If $H=S_2$, $S_3$, or $S_4$, then it follows from Theorem 2 in \cite{Lu97} that
\[
d(H\wr G)=\max\left(d(C_2\wr G),\left\lfloor\frac{d(H)-2}{n}\right\rfloor+2\right)=\max(2, d(C_2\wr G)).
\]
\end{prf}

We will now then prove that $d(F)=\order{I}$ when $\order{I}=2,3$. In the rest of this section, we will denote $S_l=\Sym(\Delta_i)$, $S_m=\Sym(\Delta_j)$ and $S_n= \Sym(\Delta_k)$. When $I$ falls into case~\ref{i:chain}, using Lemma~\ref{lem:chain} repeatedly, we deduce that $F$ is isomorphic to iterated wreath products of symmetric groups. Corollary B in \cite{LuQuick} gives that $d(F)=\order{I}$. In case~\ref{i:antichain} where $I$ is an antichain, $F$ is a direct product of symmetric groups by Lemma~\ref{lem:antichain}. This direct product has $\order{I}$ factors, and so $d(F)=\order{I}$. In the case~\ref{i:triangle}, by Lemma~\ref{lem:Generalisediswreath},
\[
F\cong \Sym(\Delta_i)\wr_{\Delta_j\times\Delta_k}(\Sym(\Delta_j)\times \Sym(\Delta_k)).
\]
Then Lemma~\ref{lem:symwreath} shows that $d(F)=d(C_2\wr_{\Delta_j\times\Delta_k}(\Sym(\Delta_j)\times \Sym(\Delta_k)))=3$. The following lemma will show that $d(F)=3$ when $I$ falls into case~\ref{i:pyramid}.

\begin{lem}
If $I$ belongs to case~\ref{i:pyramid}, then $d(F)=\order{I}$.
\end{lem}
\begin{prf}
Lemma~\ref{lem:Generalisediswreath} tells us that $F\cong S_l\wr_{\Delta_k}(S_m\wr_{\Delta_k} S_n)$. As the base group $S_m^{\Delta_k}$ of the wreath product $S_m\wr_{\Delta_k} S_n$ commutes with $S_l^{\Delta_k}$, $F\cong (S_l\times S_m)\wr_{\Delta_k} S_n$. Set $W=(S_l\times S_m)\wr_{\Delta_k} S_n$. Observe that $C_2^3$ is an image of $W$, and so $d(F)\geq3$.

We will then give a generating set of $W$. We express the base group $(S_l\times S_m)^n$ of the wreath product $W$ as $(S_l\times S_m)\times (S_l\times S_m)\times \dots \times (S_l\times S_m)$, and we write $((g_1,h_1),\dots, (g_n,h_n))f$ for elements in $W$, where $g_1,\dots, g_n$ are elements of $S_l$, $h_1,\dots, h_n$ are elements of $S_m$ and $f$ is an element of $S_n$. Corollary B in \cite{LuQuick} tells us that $d(S_m\wr S_n)=2$. Using Gasch{\"u}tz's Lemma, we have the following generators of $S_m\wr S_n$:
\begin{align*}
a=(a_1,\dots, a_n)(1\ 2) && b=(b_1,\dots, b_n)\beta,
\end{align*}
where each $a_i$ and $b_i$ is an element of $S_m$ and $\beta\in S_n$ generates $S_n$ with $(1\ 2)$. There is an element $\alpha\in S_l$ such that $\langle(1\ 2), \alpha\rangle=S_l$ and $\alpha\in A_l$. Let
\begin{align*}
&x=(((1\ 2), 1), (1,1),(1,1),\dots, (1,1))\\
&y=((\alpha, a_1), (1,a_2),(1,a_3),\dots, (1,a_n))(1\ 2)\\
&z=((1, b_1), (1,b_2),(1,b_3),\dots, (1,b_n))\beta
\end{align*}
to be elements of $W$ and $G=\langle x, y, z \rangle$. Also set 
\begin{align*}
\bar\alpha&=((\alpha, 1), (1,1),(1,1),\dots, (1,1))\\
\bar y&=((1, a_1), (1,a_2),(1,a_3),\dots, (1,a_n))(1\ 2)
\end{align*}
such that $y=\bar\alpha\bar y$. 

We wish to prove that $G=W$. To this end, it suffices to prove that $\bar\alpha, \bar y\in G$. This is because of the three following facts. The elements $\bar\alpha$ and $x$ generate one copy of $S_l$ of the base group of $W$, the wreath product $S_m\wr S_n$ can be generated by $\bar y$ and $z$, and the symmetric group $S_n$ permutes entries of the base group of $W$ transitively. 

Observe that 
\[
y^2=((\alpha, a_1a_2), (\alpha, a_2a_1),(1,a_3^2),(1,a_4^2),\dots, (1,a_n^2)).
\]
Consider the subgroup $H=\langle x, y^2\rangle$. Since the identity is at all entries of $x$ apart from the first entry, $H'=A_l\times \trivsubgp\times\trivsubgp\dots\times \trivsubgp$. We deduce that $\bar \alpha\in H'\leq G$, and so $\bar y\in G$. Hence $G=W$ and $d(F)=d(W)=3$.
\end{prf}

When $I$ falls into case~\ref{i:wrdi}, Lemma~\ref{lem:chain} and \ref{lem:antichain} give that $F\cong S_l\times (S_m\wr S_n)$. Then Theorem 2.5 in \cite{semigroup} shows that $d(F)=3$.

Combining all the above, we deduce the following lemma.
\begin{lem}\label{lem:I=2,3}
If $\order{I}=2$ or $3$, then $d(F)=\order{I}$.
\end{lem}

In the rest of this section, we will prove Theorem~\ref{thm:chap4main1}. We first determine the lower bound of $d(F)$.

\begin{lem}\label{lem:FJimage}
The direct product $C_2^I$ is a homomorphic image of $F$.
\end{lem}
\begin{prf}
We will use induction on the cardinality of $I$. Let $\order{I}=2$ and set $I=\{i, j\}$. Without loss of generality, let $i$ be a minimal element. It follows from Lemma~\ref{lem:chain} that $C^2_2$ is a homomorphic image of $F$. This establishes the base case.

Now let $n\geq2$ be a fixed but arbitrary integer. Suppose that the lemma holds whenever $\order{I}=n$. We will prove that the lemma holds when $\order{I}=n+1$. Let $j$ be a minimal element in $I$ and set $J=I\backslash\{j\}$. According to Lemma~\ref{lem:Generalisediswreath}, $F\cong \Sym(\Delta_k)\wr_{\Delta_{A(k)}} F_J$.
Hence $C_2\times F_J$ is an image of $F$. Using the induction on $F_J$, we deduce that $C_2^{I}$ is an image of $F$.
\end{prf} 

Now we turn to the upper bound of $d(F)$.

\begin{lem}\label{lem:generalisedupper}
Let the partially ordered set $I$ and the generalised wreath product $F$ be as defined at the beginning of this section. Then $d(F)\leq\order{I}$.
\end{lem}
\begin{prf}
We will use induction on the cardinality of $I$ to prove the lemma. Lemma~\ref{lem:I=2,3} tells us that the lemma holds when $\order{I}$ is $2$ or $3$. Let $n$ be a fixed but arbitrary integer with $n\geq3$. Suppose that the theorem holds whenever $2\leq \order{I}\leq n$. We will prove that the lemma holds when $\order{I}=n+1$. 

Suppose that $I$ has a unique minimal element $m$. Then set $K=I\backslash\{m\}$, and $K$ is an ancestral subset. It follows from Lemma~\ref{lem:Generalisediswreath}\ref{i:unique} that $F\cong \Sym(\Delta_m)\wr_{\Delta_{A(m)}} F_K$, and the action of $F_K$ on $\Delta_{A(m)}=\Delta_K$ is faithful. Lemma~\ref{lem:transitive} tells us that $F_K$ is transitive on $\Delta_K$. Then Lemma~\ref{lem:symwreath} shows that $d(F)=\max(2, d(C_2\wr F_K))$. By the inductive hypothesis, we have
\[
d(C_2\wr F_K)\leq d(C_2)+d(F_K) =\order{K}+1 =\order{I}.
\]
Then $d(F)\leq\max(2, \order{I})=\order{I}$.

Suppose that $m, n$ are two different minimal elements in $I$. Then set $K=I\backslash\{m,n\}$. We observe that $K$ is ancestral and $F_K$ is the generalised wreath product constructed from $(\Sym(\Delta_i),\Delta_i)_{i\in K}$. Lemma~\ref{lem:GeneralisedD} shows that $F=\langle D_i:i\in I\rangle$ and $F_K=\langle D_i^K:i\in K\rangle$. It follows from
Lemma~\ref{lem:barFKiso} that $\bar F_K=\langle D_i:i\in K\rangle$. Therefore
\[
d(\langle D_i:i\in K\rangle)=d(\bar F_K)=d(F_K)=\order{K}.
\]
On the other hand, by Lemma~\ref{lem:Hinormal}\ref{i:HicrossHi} and \ref{lem:Hidirectproduct}, we obtain that $\langle H_m, H_n\rangle$ is isomorphic to $ \Sym(\Delta_m)^{\Delta_{A(m)}}\times\Sym(\Delta_n)^{\Delta_{A(n)}}$. Hence, for $D_m$ and $D_n$ as defined in Equation~\eqref{eq:DiJdefinition}, $\langle D_m, D_n \rangle\cong \Sym(\Delta_m)\times \Sym(\Delta_n) $. Then $d(\langle D_m, D_n \rangle)=2$. Therefore,
\[
d(F)=d(\langle D_i:i\in I\rangle)\leq d(\langle D_m, D_n \rangle)+d(\langle D_i:i\in K\rangle) = 2+\order{K}=\order{I}.
\]
The induction is completed and we have proved the lemma.
\end{prf}

{\small


\begin{thebibliography}{99} \setlength{\itemsep}{-0.2ex}

\bibitem{Ana}
M. Anagnostopoulou-Merkouri, R.~A. Bailey, P.~J. Cameron. Permutation groups, partition lattices and block structures. \textit{Forum Math. Sigma}, {\bf 13} (2025), e180, pp. 1–32.

\bibitem{semigroup}
J. Ara\'ujo, W. Bentz, J. Mitchell, C. Schneider. The rank of the semigroup of transformations stabilising a partition of a finite set. \textit{Math. Proc. Cambridge Philos. Soc.}, {\bf 159} (2015), no.~2, pp. 339–353.

\bibitem{Generalised}
R. A. Bailey, C. E. Prager, C. A. Rowley, T. P. Speed. Generalized wreath products of permutation groups. \textit{Proc. Lond. Math. Soc.(3)}, {\bf 1} (1983), pp. 69–82.

\bibitem{Bond}
  Ievgen V. Bondarenko, ``Finite generation of iterated wreath
  products,'' \textit{Arch.\ Math.}~\textbf{95} (2010), pp. 301–308.

\bibitem{LuDallaVolta2}
F. Dalla~Volta, A. Lucchini. Finite groups that need more generators
than any proper quotient. \textit{J. Austral. Math. Soc.}, Series A, {\bf 64} (1998), pp. 82–91.


\bibitem{LuQuick}  
  J. Lu, M. Quick. Generation of iterated wreath products constructed from alternating, symmetric and cyclic groups. \textit{Internat. J. Algebra Comput}, {\bf 35} (2025), no.~5, pp. 713–731.
  
  \bibitem{Lu97}
A. Lucchini. Generating wreath products and their augmentation ideals. \textit{Rend. Sem. Mat. Univ. Padova}, {\bf 98} (1997), pp. 67–87.

\bibitem{LuMe}
A. Lucchini, F. Menegazzo. Generators for finite groups with a unique minimal normal subgroup. \textit{Rend. Sem. Mat. Univ. Padova}, {\bf98} (1997), pp. 173–191.

\bibitem{LuTha}
A. Lucchini, D. Thakkar. The minimum generating set problem. \textit{J. Algebra}, {\bf 640} (2024), pp. 117–128.

\bibitem{MRQ2}
  Martyn Quick, ``Probabilistic generation of wreath products of
  non-abelian finite simple groups,~II,''
  \textit{Internat.\ J. Algebra Comput.}~\textbf{16} (2006), no.~3, pp. 493–503.


\bibitem{Wiegold}
  James Wiegold, ``Growth sequences of finite groups,''
  \textit{J. Aus.\ Math.\ Soc.}~\textbf{17} (1974), pp. 133–141.





\end{thebibliography}
\end{document}